%% file: Cex-KM-TJPeters-agg-arX.tex
\documentclass[11pt]{article}
\usepackage{amsfonts,amsmath,amsthm,mathrsfs,amssymb,graphicx}
\usepackage{ dsfont}
\usepackage{stmaryrd}
\usepackage{appendix}
\usepackage{subfigure}
\usepackage{url}
\usepackage{hyperref}
\usepackage{multicol}
\usepackage{verbatim}

\newtheorem{theorem}{Theorem}[section]
\newtheorem{cor}{Corollary}[section]
\newtheorem{definition}{Definition}[section]

\newtheorem{lemma}{Lemma}[section]
\newtheorem{claim}{Claim}[lemma]
\newtheorem{remark}{Remark}[section]

\begin{document}

\title{Brief Article}
\author{The Author}

\begin{center}
{\bf Exact Computation for Existence of a Knot Counterexample}\\
{\bf T. J. Peters, K. Marinelli, University of Connecticut}\\

{\bf Abstract}\\
\end{center}

Previously, numerical evidence was presented of a self-intersecting B\'ezier curve having the $\mbox{unknot}$ for its control polygon.  This numerical demonstration resolved open questions in scientific visualization, but did not provide a formal proof of self-intersection.   An example with a formal existence proof is given, even while the exact self-intersection point remains undetermined.

\section{Introduction to Existence Condition for Self-intersection}
\label{sec:ex}

The formal proof of an unknotted control polygon with a self-intersecting B\'ezier curve (See Definition~\ref{def:c}.) appears in Theorem~\ref{thm:hasselfi}, but illustrative images are shown in Figure~\ref{fig:xselfi}.  The progression, from left to right, shows `snap shots' of a piecewise linear (PL) curve being continuously perturbed, generating new B\'ezier curves at each instant.  The top-most point is the only one perturbed linearly.  The initial (left) and final (right) B\'ezier curves are shown to have differing knot types, so there must be an intermediate B\'ezier curve with a self-intersection (In the middle image, a red dot approximates a neighborhood containing that intersection).  Relevant definitions follow.

\hspace*{-15ex}
\begin{figure}[h!]
\centering
    {
   \includegraphics[height=6cm]{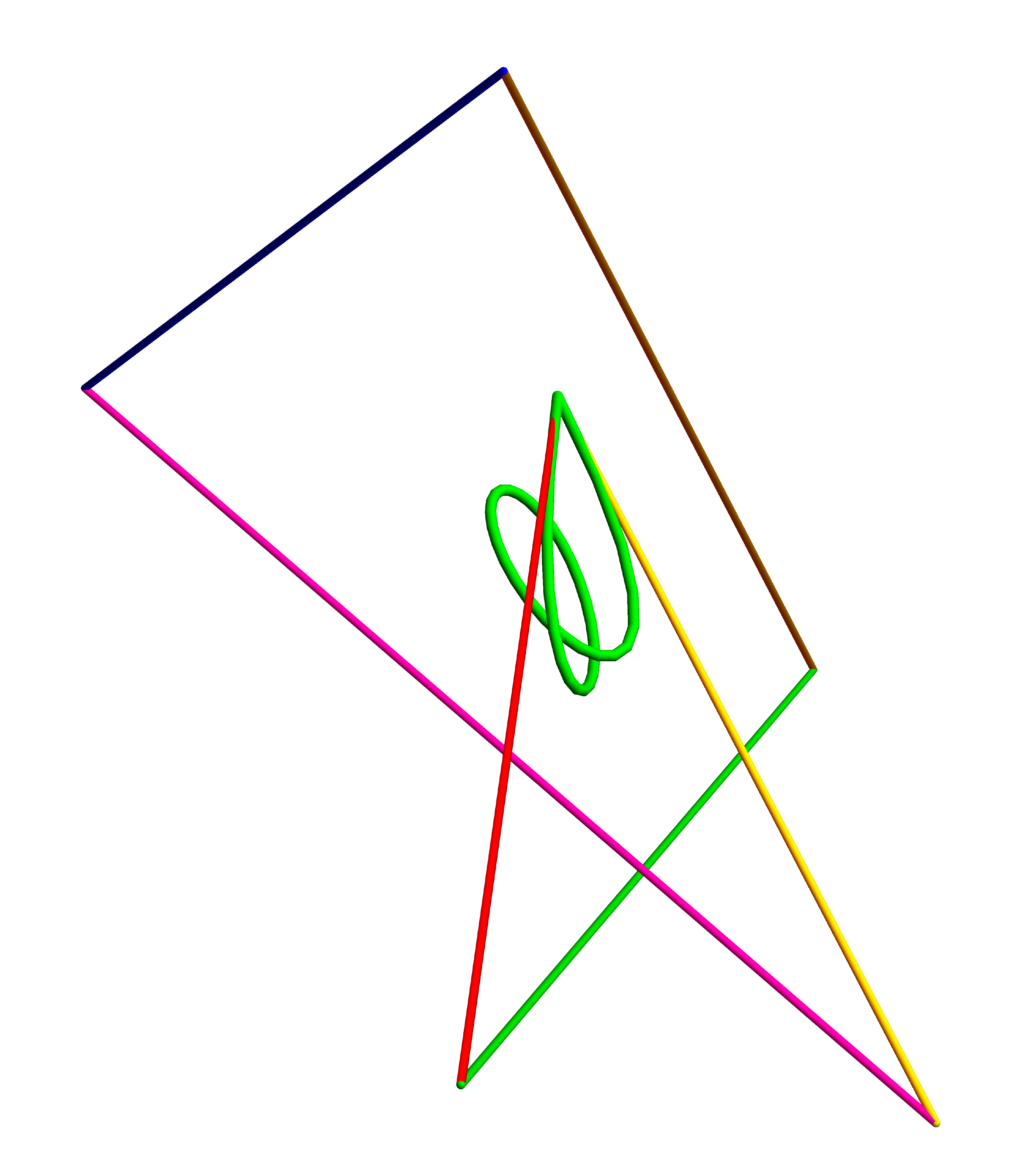}
    }
    {
   \includegraphics[height=6cm]{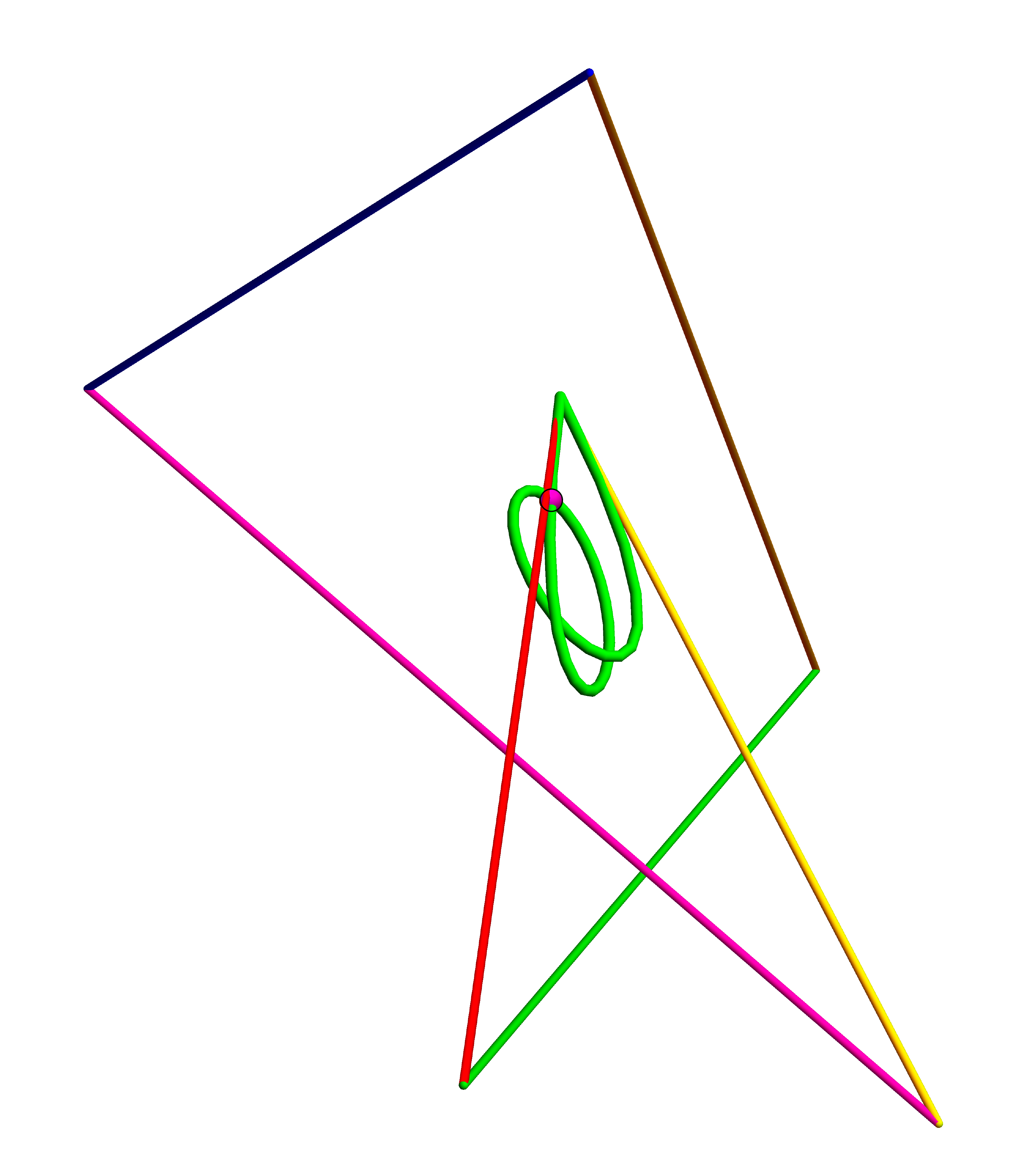}
   \label{fig:unkview}
    }
    {
   \includegraphics[height=6cm]{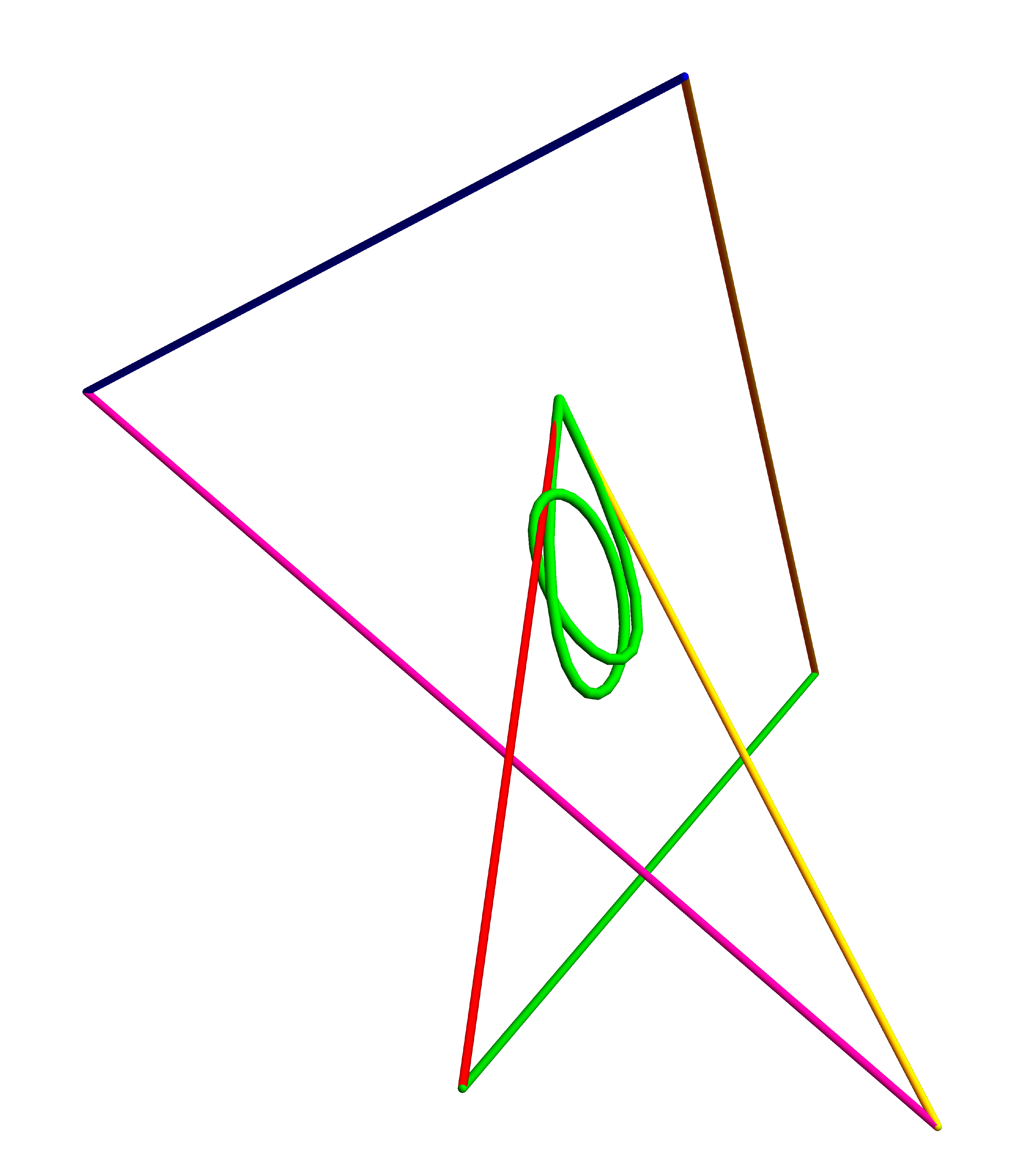}
    }
\vspace{-2ex}
    \caption{Existence of Self-intersection}
    \label{fig:xselfi}
\end{figure}

\section{Data for the Example}
\label{sec:dataex}

Let 
$$ (0, 9, 20), (-15, -95, -50), (40, 80, -20), \mathbf{(-10, -60, 58)},
(-60, 30, 20), (40, -60, -60), (0, 9, 20)$$
 be the ordered list of vertices for the initial PL knot, denoted by $K_0$, where the integer values support exact computation.   In Figure~\ref{fig:xselfi}, the vertex $(0, 9, 20)$ is the initial and final point of the B\'ezier curve shown in each of the three images (in green).  The remaining vertices are arranged counterclockwise.
A single Reidemeister move of vertex (40, -60, -60) 
shows that $K_0$ is the unknot.    
 
The vertex in bold font is perturbed linearly to $\mathbf{(10, -60, 58)}$ to obtain the final PL knot, denoted by $K_1$, with all other vertices remaining fixed.   The perturbation generates uncountably many new control polygons and associated B\'ezier curves~\footnote{Author T. J. Peters thanks the referee of the previously cited numerical work \cite{JL2012} for a suggestion, that was similar in spirit, to generate the example shown here.}.  The B\'ezier curves corresponding to $K_0$  and $K_1$ are denoted by  $\beta_0$ and $\beta_1$, respectively.   Both $K_0$ and $K_1$ were produced from visual experiments.


\section{Mathematical Preliminaries}
\label{sec:mathpre}

Some relevant mathematical definitions are summarized.

\begin{definition}
\label{def:knot}
A subspace of $\mathbf{R}^3$ is called a knot \cite{Armstrong1983} if it is homeomorphic to a circle.
\end{definition}

Homeomorphism does not distinguish between different embeddings.  The stronger equivalence by ambient isotopy distinguishes embeddings and is fundamental in knot theory.

\begin{definition}
Let $X$ and $Y$ be two subspaces of $\mathbb{R}^3$.  A continuous function
\[ H:\mathbb{R}^3 \times [0,1] \to \mathbb{R}^3 \]
is an {\bf ambient isotopy} between $X$ and $Y$ if $H$ satisfies the following conditions:
\begin{enumerate}
\vspace{-0.1in}
\item $H(\cdot, 0)$ is the identity,
\vspace{-0.1in}
\item $H(X,1) = Y$, and
\vspace{-0.1in}
\item $\forall t \in [0,1], H(\cdot,t)$ is a homeomorphism from
$\mathbb{R}^3$ onto $\mathbb{R}^3$.
\vspace{-0.1in}
\end{enumerate}
\label{def:aiso}
\end{definition}

\begin{definition}\label{def:c}
Denote $\mathcal{C}(t)$ as the parameterized B\'ezier curve of degree $n$ with control points $P_m \in \mathbf{R}^3$, defined by
$$
\mathcal{C}(t)=\sum_{i=0}^{n}{B_{i,n}(t)P_i}, \hspace{2ex}
t\in[0,1]
$$
where $B_{i,n}(t) = \left(\!\!\!
  \begin{array}{c}
	n \\
	i
  \end{array}
  \!\!\!\right)t^i(1-t)^{n-i}$.
\end{definition}

The curve $\mathcal{P}$ formed by PL interpolation on the ordered list of points $P_0,P_1,\ldots,P_n$ is called the {\em control polygon} and is a PL approximation of $\mathcal{C}$.

The work presented was partially motivated by development of wireless data gloves. The smooth B\'ezier representations are manipulated by the gloves, while the illustrative computer graphics are created by PL approximations derived from the control polygon.  Topological fidelity between these representations is of interest to provide reliable visual feedback to the user.

\section{Related Work}
\label{sec:relw}

This article arose from a question posed by a referee from a previous numerical example~\cite{JL2012}.

The mathematical objects of study here are B\'ezier curves.  The canonical references~\cite{G.Farin1990} (any edition) and \cite{Piegl} focus on approximation and modeling, with less attention to associated topological properties.  Relations between a smooth curve and its PL approximation are dominant in computer graphics~\cite{Fvd90}, but the topological aspects are often ignored.   One prominent property is that any B\'ezier curve is contained in the convex hull of its control points~\cite{G.Farin1990}, while recent enhancements have been shown~\cite{PetersWu}, where the containing set is a subset of the convex hull. 

The \emph{push} receives prominent attention by R. H. Bing~\cite{Bing1983} as a fundamental tool in developing ambient isotopies in $\Re^2$.  The perturbation in this article is the trivial extension of a push to $\Re^3$.

The preservation of topological characteristics in geometric modeling and graphics has become of contemporary interest \cite{Amenta2003, L.-E.Andersson2000, Andersson1998, Chazal2005, ENSTPPKS12, bez-iso,KiSi08, JL-isoconvthm}.  Sufficient conditions for a homeomorphism between a B\'ezier curve and its control polygon have been studied \cite{M.Neagu_E.Calcoen_B.Lacolle2000}, while topological differences have also been shown \cite{Bisceglio, JL2012, Piegl}.   

Topological aspects are relevant in  `molecular movies' \cite{MMovies}.   Sufficient conditions were established for preservation of knot type~\cite{TJP08}  during dynamic visualization of ongoing molecular simulations \cite{WeMi06}.  Interest by bio-chemists in using hand gestures to interactively manipulate these complex molecular images prompts related topological considerations for data gloves~\cite{CG,Manus,VM}. 

The beautiful visual studies of knot symmetry~\cite{Carlo} provided an example of the $7_4$ knot that was helpful in the initial visual studies performed to create the example presented here.   

Exact computation has some advantages over floating point arithmetic~\cite{culver2004exact,kettner2008classroom}, particularly regarding loss of precision in finite bit strings, with alternative views expressed~\cite{Jiang2006}.  Some \emph{ad hoc} techniques are often invoked to provide adequate bit length for full precision relative to a particular computation, as has been done here.

 \section{Exact Computation for Subdivision}
 \label{sec:exsub}
 
The de Casteljau algorithm \cite{G.Farin1990} is a subdivision method for B\'ezier curves.  The algorithm recursively generates control polygons that more closely approximate the curve~\cite{G.Farin1990} under Hausdorff distance~\cite{J.Munkres1999}.   Figure~\ref{fig:sub-ex} shows the first step of the de Casteljau algorithm with an input value of $\frac{1}{2}$ on a cubic B\'ezier curve.  The initial control polygon $P$ is used as input to generate local PL approximations,  $P^1$ and $P^2$, as Figure~\ref{fig:dec3} shows.  The result is to subdivide the original curve into two subsets, with the new control polygons, $P^1$ and $P^2$, respectively.  

\begin{figure}[h!]
\centering
        \subfigure[Subdivision process]
   {   \includegraphics[height=2.7cm]{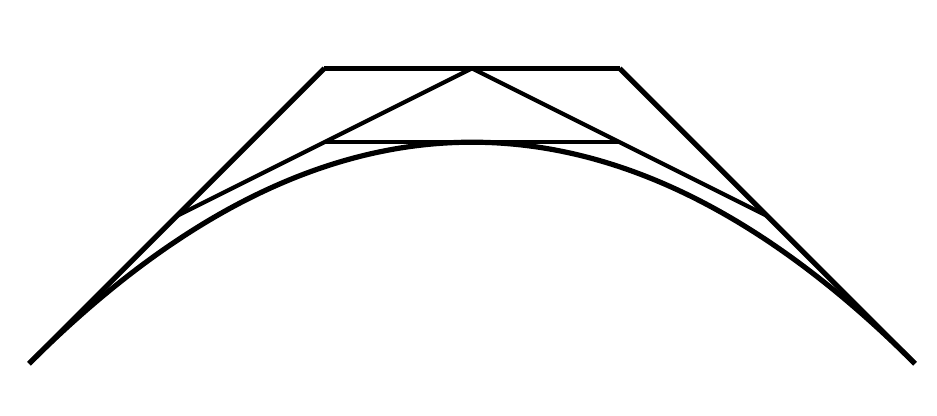}\label{fig:dec1} }
              \subfigure[$P^1$ and $P^2$]
   {   \includegraphics[height=2.7cm]{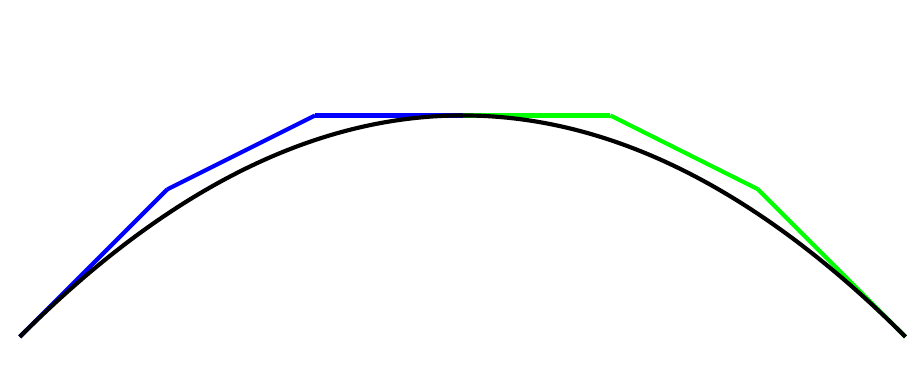}\label{fig:dec3} }
\caption{A subdivision with parameter $\frac{1}{2}$}\label{fig:sub-ex}
 \end{figure}

Figure~\ref{fig:dec1} has the initial control polygon $\mathcal{P}$ as the three outer edges.   The next step of this subdivision produces 4 edges by selecting the midpoint of each edge of $P$ and connecting these midpoints, producing 
4 edges closer to the B\'ezier curve in Figure~\ref{fig:dec1}.   Recursion continues until a tangent to the B\'ezier curve \cite{G.Farin1990} is created. The union of the edges from the final step then forms two new PL curves, as shown in Figure~\ref{fig:dec1}. Termination is guaranteed over finitely many edges.

For $i$ iterations, the subdivision process has generated $2^i$ PL sub-curves, each being a control polygon for part of the original curve \cite{G.Farin1990}, which is a \textbf{sub-control polygon}\footnote{Note that by the subdivision process, each sub-control polygon of a simple B\'ezier curve is open.}, denoted by $\mathcal{P}^k$ for $k = 1, 2, 3, \ldots, 2^i$.   For a curve initially defined by $n + 1$ control points, each $\mathcal{P}^k$ has $n+1$ points and their union $\bigcup_k  \mathcal{P}^k$ forms a new PL curve that converges in Hausdorff distance to the original B\'ezier curve.  The B\'ezier curve defined by $\bigcup_k \mathcal{P}^k$ is exactly the same as the original B\'ezier curve~\cite{Lane_Riesenfeld1980}. 


\begin{remark}
A B\'ezier curve is contained within the convex hull of its control 
points~ \cite{G.Farin1990}. 
\label{rem:convhull} 
\end{remark}

Remark~\ref{rem:convhull} also applies to each sub-control polygon, as will be extensively used.

The stick knots of Section~\ref{sec:dataex} are refined via the DeCasteljau subdivision algorithm.
Each refinement stage in the DeCasteljau algorithm uses a series of averages for the vertices of the form $\frac{(a+b)}{2}$.   In order to retain integer valued vertices throughout all computed subdivisions, the vertices will be scaled  by $2^m$, for an appropriately chosen $m$. The existence of self-intersections in the B\'ezier curves is preserved under scaling.  Each subdivision iteration has one division by 2 for each degree $d$.  In our examples, we invoke 
 $\ell$ levels of subdivision,  so that $m = \ell \ast (d + 1) +1$.

\section{Ambient Isotopy Between Stick Knots}
\label{sec:aisticks}

The two curves, $K_0$ and $K_1$ will be shown to be the unknot.

\begin{lemma}
The curve $K_0$ is the unknot.
\end{lemma}

\begin{proof}
The necessary demonstration that $K_0$ is simple is elementary over the six edges.   Establishing that no consecutive edges are collinear is by comparing slopes of projections into the $XY$-plane.  The non-consecutive pairs of edges have separating hyperplanes. \hspace{5ex} $\Box$

\end{proof}

\vspace{1ex}

\begin{lemma}
The curves $K_0$ and $K_1$ are ambient isotopic.
\end{lemma}

\begin{proof}
The single perturbation to create $K_1$ is similar to a push~\cite{Bing1983} and incurs no self-intersections with other segments.  It is easily extended to an ambient isotopy having compact support.   \hspace{5ex} $\Box$
\end{proof}

\vspace{1ex}

\begin{cor}
The curve $K_1$ is the unknot.
\end{cor}

\section{Subdivision Analysis}
\label{sec:subdivan}

The points from the 4th iteration of the DeCastlejau algorithm, using a subdivision value of $1/2$ are listed in Appendix~\ref{app:subdata}.   These can be verified simply by executing the algorithm on the initial data for $K_0$ and $K_1$, with the provision that the implementation relies upon exact arithmetic.

It was observed that each sub-control polygon in Appendix~\ref{app:subdata} had the property that its control points were strictly monotone in one of its coordinate values.  This is annotated in the appendix by abbreviations noting the co-ordinates in which strict monotonicity was observed.  In some cases, this occurs for all three co-ordinates.  This observation greatly facilitated proof techniques used.

Let $K_{0,4}$ denote the PL curve formed from the 4th subdivision on $K_0$ and similarly denote $K_{1,4}$ for $K_1$.
Within each of $K_{0,4}$ and $K_{1,4}$, there are 16 sub-control polygons.  The rest of the paper proceeds by showing that $K_{0,4}$ is the unknot and is ambient isotopic to $\beta_0$ and that $K_{1,4}$ is the unknot and is ambient isotopic to $\beta_1$.  As the convex hulls of the sub-control polygons become central, their images are shown in Figure~\ref{fig:chulls}. 

\begin{figure}[h!]
\centering
       \subfigure[$K_0$]
       {   \includegraphics[height=6cm]{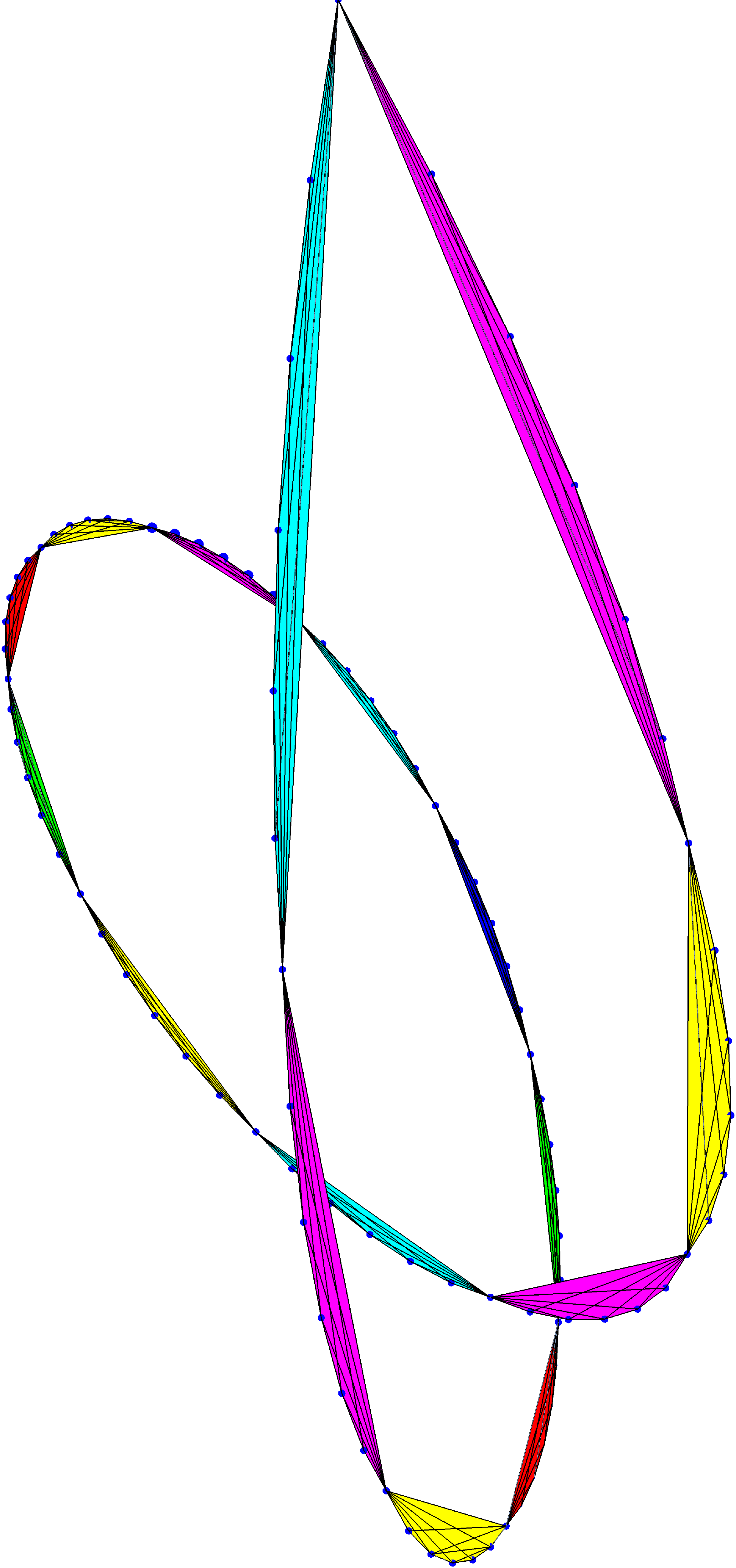}
          \label{fig:sub0-4} }  \hspace*{15ex}
              \subfigure[$K_1$]
   {   \includegraphics[height=6cm]{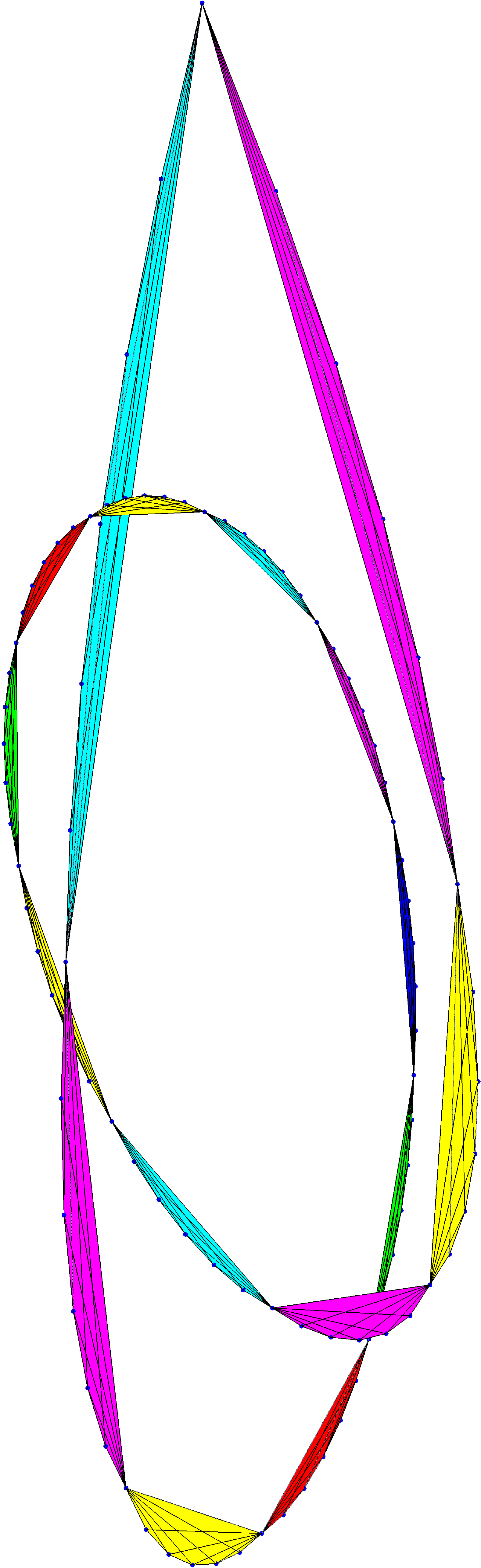}
  \label{fig:sub1-4} }
\caption{Convex Hulls -- Fourth subdivision (parameter $\frac{1}{2}$)}
\label{fig:chulls}
\end{figure}

\pagebreak

\begin{lemma}
\label{lem:both4-simp}
Each of $K_{0,4}$ and $K_{1,4}$ are simple.
\end{lemma}

\begin{proof}
First consider $K_{0,4}$.  Let $K_{0,4,i}$, for $i = 0, 1, \ldots 15$, denote the 16 sub-control polygons for $K_{0,4}$.  Each $K_{0,4,i}$, for  $i = 0, 1, \ldots 15$, is simple, because of the strict monotonicity in at least one co-ordinate.  The $K_{0,4,i}$, for $i = 0, 1, \ldots 15$, are pairwise disjoint, as shown by computing the convex hull for each and establishing these convex hulls are disjoint.  Again, any ambiguity that could arise from floating point arithmetic is avoided by exact computations.
 
The argument for $K_{1,4}$ follows the same pattern.  \hspace{5ex} $\Box$
\end{proof}


The derivative of a B\'ezier curve $\cal{C}$ is expressed as 

$$ \mathcal{C}^{'}(t)=\sum_{i=0}^{n} \binom{n}{i} t^i (1-t)^{n-i} (P_{i+1}-P_i), i \in \{1, \ldots, {n-1}\}.$$



\begin{lemma}
\label{lem:both4-mono}
Each of $\beta_0$ and $\beta_1$ are strictly monotonic in the same co-ordinate as their corresponding control polygons and are simple.
\end{lemma}

\begin{proof}
The strict monotonicity in some co-ordinate for each sub-control polygon implies that the partial derivative in the $x,y$ or $z$ variable will be positive, implying simplicity for that segment of the B\'ezier curve.  Since each B\'ezier segment is contained in the convex hull of its sub-control polygon and since those convex hulls are pairwise disjoint, the simplicity follows.   \hspace{5ex} $\Box$
\end{proof}

\vspace{1ex}

\begin{cor}
\label{cor:monobez}
If a control polygon is strictly monotonic in some co-ordinate, then its associated B\'ezier curve is also.
\end{cor}

Subdivision need not preserve knot type, so we establish the knot types for $K_{0,4}$ and $K_{1,4}$.
\begin{lemma}
The curve $K_{0,4}$ is the unknot.
\end{lemma}

\begin{proof}
The projection onto the $XY$--$plane$ of $K_{0.4}$, has three pairs of crossings, as again, can be verified by exact computation with some co-ordinate values being represented exactly as rational numbers.  For the $z$ co-ordinates, a bounding integer value is given to these exact rational representations, for ease of comparison in determining the crossings relative to the $XY$--$plane$:
   \vskip 0.1in
   overcrossing: $( (70600874219532518400/90998355737), (-331936500725210686464/90998355737) )$,
   \newline
   with $z$ values $(346821834258400839680/90998355737) > -179855360$
   \vskip 0.1in
   
   overcrossing: $( (-2606810091676070400/2228701007), (-141298996572704411136/15600907049) )$,
   \newline
   with $z$ values $ (-695629074309606400/821100371) <  -847191279 < -2734161920$
   
   \vskip 0.1in
    undercrossing: $( (70736463317966883840/46441845451), (-2183313901438239630336/232209227255) )$,
    \newline
    with $z$ values $ (-312474348049921062912/46441845451) < -6728293094 < -5234410880$
\vskip 0.1in

The two consecutive overcrossings permit conclusion of the unknot.   \hspace{5ex} $\Box$
\end{proof}

\vspace{1ex}

\begin{lemma}
The curve $K_{1,4}$ is the trefoil.
\end{lemma}

\begin{proof}

The projection onto the $YZ$--$plane$ of $K_{1.4}$, has three pairs of crossings, from exact computation:

   \vskip 0.1in
    undercrossing: $( (-2321511588927897600/2778711187),  (-12860064917297823744/2778711187) )$,
    \newline
    with $z$ values $ (8352902508791726080/2778711187) < 3006034794 <  4026531840$
   \vskip 0.1in 
   
   overcrossing: $( (-8217249783839948800/7079707793), (-58835463893254373376/7079707793) )$,
   \newline
   with $z$ values $ (-24693447116718080/7079707793) > -34879189 > \\
   \hspace*{25ex} -1547779857 > (-10957829115791278080/7079707793)$
   
   \vskip 0.1in
  undercrossing: $( (2736710937161450496/968030497), (-222045650166834551808/24200762425) )$,
  \newline
  with $z$ values $ (-6443291820066594816/968030497) < -6656083501 <  -5348303360$.

These three pairs of alternating crossings permit conclusion of the trefoil.   \hspace{5ex} $\Box$
\end{proof}

\vspace{1ex}

\section{Ambient Isotopy for B\'ezier Curves}
\label{sec:aibez}

We show that $\beta_0$ is ambient isotopic to $K_{0,4}$ and that $\beta_1$  is ambient isotopic to $K_{1,4}$.  Both proofs proceed by showing that each sub-control polygon is ambient isotopic to its associated B\'ezier curve.  We show this for one sub-control polygon, as representative.  Let $K_{0,4,i}$ be one such control polygon and let $\beta_{0,i}$ be its corresponding B\'ezier curve segment.

\begin{theorem}
Each  $K_{0,4,i}$ is ambient isotopic to its corresponding B\'ezier curve, denoted as $\beta_{0,i}$.
\end{theorem}

\begin{proof}
Without loss of generality assume $K_{0,4,i}$ is strictly monotonic in its $x$ co-ordinate, which implies that $\beta_{0,i}$ is also strictly monotonic in its $x$ co-ordinate by Corollary~\ref{cor:monobez}.  For each $p \in K_{0,4,i}$, let $\Pi_p$ be the plane at $p$ that is parallel to the $YZ$-plane.  By the indicated strict monotonicity, $\Pi_p \cap K_{0,4,i} = \{p\}$, a unique point.  Similarly, by the same strict monotonicity on $\beta_{0,i}$, the intersection $\Pi_p \cap \beta_{0,i}$ has at most one point, and the connectivity of 
$\beta_{0,i}$ (as the continuous image of a connected subset of $[0,1]$) implies that this intersection is non-empty.  For each $p \in K_{0,4,i}$, let $q_p = \Pi_p \cap \beta_{0,i}$.  To construct the ambient isotopy, consider the line segment between each $p$ and $q_p$.  These line segments will not intersect, due to the strict monotonicity.  The remaining details to construct an ambient isotopy of compact support are standard and left to the reader. \hspace{5ex} $\Box$
\end{proof}

\pagebreak

\begin{cor}
\label{cor:aibez0}
The B\'ezier curve $\beta_0$ is ambient isotopic to its control polygon $K_{0,4}$, the unknot.
\end{cor}

\begin{proof}
The existence of an ambient isotopy within each convex hull of $K_{0,4,i}$, the lack of intersection between convex hulls $K_{0,4,i}$ and $K_{0,4,j}$ for $i \neq j$ (except at the end points when $j = i + 1$, where those end points remain fixed) and composition of the ambient isotopies on each sub-control polygon conclude this proof.
 \hspace{5ex} $\Box$
\end{proof}
\vspace{1ex}

\begin{cor}
\label{cor:aibothbez3}
The B\'ezier curve $\beta_1$ is ambient isotopic to its control polygon $K_{1,4}$, the trefoil.
\end{cor}

\begin{theorem}
\label{thm:hasselfi}
The ambient isotopy between $K_0$ and $K_1$ generates a B\'ezier curve with a self-intersection.
\end{theorem}

\begin{proof}
The perturbation of the control point $(-10, 60, 58)$ to $(10, 60, 58)$ creates a homotopy on the B\'ezier curve $\beta_0$ to $\beta_1$, defined over the interval $[0,1]$.  If for all values of $t \in [0,1]$, the perturbed B\'ezier curve was simple, then the homotopy would be an isotopy between $\beta_0$ and $\beta_1$.  However, since  $\beta_0$ and $\beta_1$ have differing knot types, this cannot be.  So, for some value 
\~{t} $\in [0,1]$, the corresponding B\'ezier curve must have a self-intersection.
 \hspace{5ex} $\Box$
\end{proof}

\section{More Aggressive Enclosures}
\label{sec:agg}

In the spirit of optimal enclosures for splines~\cite{PetersWu}, a modification is shown to the already presented proof of isotopy for the B\'ezier trefoil at the 4th subdivision.   The benefit is an explicit construction of an isotopy of the trefoil at the 3rd subdivision.   The modification presented is, admittedly, \emph{ad hoc} for the given data, but suggests the possibility of using the broader generalizations~\cite{PetersWu}.

In the construction for the 4th subdivision, it was shown that the convex hulls generated were pairwise disjoint, except, possibly, at a single control point.  For the 3rd subdivision, this property failed only for one pair of convex hulls, which did not have any shared control points.  It was possible to modify one of these convex hulls to a more aggressive enclosure to be disjoint from the other convex hull.   An outline of the modification follows.

The convex hull chosen for modification is determined by the control polygon $K_{1,3,7}$, as listed in the Appendix.  Let $H_{1,3,7}$ denote the convex hull of $K_{1,3,7}$ and let $c$ denote the B\'ezier segment within $K_{1,3,7}$, re-parameterized as $c: [0,1] \rightarrow \Re^3$, where 
$$c(1/2) = (4399876800, -4859733312, -855503360).$$ 
Let $p_L$ denote the initial point of $K_{1,3,7}$ and let $p_R$ denote the final point of $K_{1,3,7}$.

Two half-spaces will be constructed, with their definitions being  interdependent.  First, consider the line $\ell_L$ defined by $c(1/2)$ and $p_L$.  A closed half-space ${\mathcal{H}}_L$  will be chosen to contain $\ell_L$.  There are uncountably many normals to $\ell_L$, but the vector $(1, 68748075/151430805, 0)$ is selected to define the separating plane for ${\mathcal{H}}_L$, with this plane denoted by $\Pi_L$.   Let ${\mathcal{H}}_L$ denote the closed half-space of all points with non-negative evaluation according to the planar equation defining  
$\Pi_L$.  Similarly, let  ${\mathcal{H}}_R$  be determined by the line containing $c(1/2)$ and $p_R$, with a normal chosen as $(10, -57032750/60642987, 0)$ and an associated separating plane denoted by $\Pi_L$.  The two normals were chosen by numerical experiments to define the set $\mathcal{E}$ by
\begin{center} 
 $\cal{E}$  $= H_{1,3,7} \cap ({\mathcal{H}}_L \cup {\mathcal{H}}_R).$ 
\end{center}

\pagebreak

\begin{claim}
The set $K_{1,3,7}$ is contained in $\cal{E}$.  
\label{claim:cptsin}
\end{claim}

\begin{claim}
The set $\cal{E}$ is disjoint from all the convex hulls for
$K_{1,3,i}, i = 0, \ldots, 6$, except at the end points $p_L$ and $p_R$. 
\label{claim:pwdis}
\end{claim}

The geometric calculations necessary to verify the two claims are elementary.

The well-known \emph{variation diminishing property}~\cite{Lane_Riesenfeld1980}  is fundamental for B\'ezier curves, where the monograph~\cite{G.Farin1990} serves as a convenient reference.

\begin{theorem}
Variation Diminishing Property for B\'ezier Curves:  A B\'ezier curve $\mathbf{B}$ has no more intersections with any plane than occurs between that plane and the control polygon of $\mathbf{B}$
\label{thm:vdpbez}
\end{theorem} 

\begin{lemma}
The curve $c$ is contained within $\cal{E}$.
\label{lem:cin}
\end{lemma}

\begin{proof}
The set $\Pi_L \cap K_{1,3,7}$ has cardinality two, as can be shown directly.
Similarly, the set $\Pi_R \cap K_{1,3,7}$ has cardinality two.
Invoking Theorem~\ref{thm:vdpbez}, yields 
\begin{center}
$\Pi_L \cap c = \{p_0, c(1/2)\}$ and
$\Pi_R \cap c = \{c(1/2), p_1\}$. 
\end{center} 
It can easily be shown that $c(1/4)$ and $c(3/4)$ both lie within $\cal{E}$ and are not equal to each other or to any of the points $ p_0, c(1/2)$, or $p_1$.  Connectivity of $c$ suffices to conclude the proof.   \hspace{5ex} $\Box$
\end{proof}
  
\begin{cor}
The control polygon for the 3rd subdivision is ambient isotopic to the trefoil.
\label{cor:n3tref}
\end{cor}

\begin{proof}
For the 3rd subdivision, the coordinate monotonicity property prevails.  The containment of $c$ within $\cal{E}$ and the disjoint properties shown relative to $\cal{E}$ and the other convex hulls are sufficient to lead to minor modifications of the proof of Corollay~\ref{cor:aibez0} to construct a similar isotopy for this 3rd subdivision.   \hspace{5ex} $\Box$
\end{proof}

\begin{remark}
Corollary~\ref{cor:n3tref} permits the explicit construction of an isotopy.  However, once it is known that $\beta_1$ is the trefoil, it can be shown by direct calculations that each of control polygons $K_{1,1}$ and $K{1,2}$ are also the trefoil, implying the \emph{existence} of ambient isotopies with $\beta_1$, but that method does not specifically construct the isotopies.
\label{re:n1}
\end{remark}

\section{Conclusion and Future Work}
\label{sec:concandfu}

An example is shown of a closed B\'ezier curve and its control polygon which are both the unknot.  A perturbation of one vertex is shown to cause the perturbed B\'ezier curve to be the trefoil, while its control polygon remains the unknot.  This transition demonstrates that there must have been an intermediate state where a self-intersection occurred in the transforming B\'ezier curves.  These topological differences between the mathematical representation and  its PL approximation are of interest in computer graphics and animation.  Visual evidence previously existed, but the only supportive mathematics relied on floating point arithmetic which left open the question of whether an intermediate intersection could be rigorously proven.  The formal proof presented relies on the implementation of exact, integer computations to resolve that open question.

Comparisons between the exact techniques and certifiable numeric methods merit further consideration, as it may often be of interest to specify a neighborhood in which a self-intersection is known to exist.  The present example was quite carefully constructed to show the self-intersection and robust numerical methods are likely to have broader scope for applications.

\vspace{1ex}

\bibliographystyle{plain}
\small
\bibliography{tjp-new}

\pagebreak

\appendix
\section{Subdivision Data}
\label{app:subdata}

The following two tables list the vertices for both the original and perturbed control polygons, each after 4 subdivisions.

The strict monotonicity of specific coordinates is indicated for each subdivided control polygon.

In Table I for subdivision of $K_0$ the sixteen control polygons are denoted by P[0], $\dots$, P[15] and all consecutive control polygons are separated by half-planes.  Between P[0] and P[1] the separating plane is the $XZ$ plane at the $y$-value of the last point of P[0].  Between P[1] and P[2] the separating plane is the $XZ$ plane at the $y$-value of the last point of P[1].  For all others, it is a $YZ$ plane at the $x$ value of the shared subdivision point.  

In Table II for subdivision of $K_1$ the sixteen control polygons are denoted by Q[0], $\dots$, Q[15] and all consecutive control polygons are separated by half-planes.  Between Q[9] and Q[10] the separating plane is the $XZ$ plane at the $y$-value of the last point of P[0].  Between Q[14] and Q[15] the separating plane is the $XZ$ plane at the $y$-value of the last point of P[14]. For all others, it is a $YZ$ plane at the $x$ value of the shared subdivision point.  

A notational change for points occurs in Appendix~\ref{app:subdata}.  In preceding narrative, points are denoted by parentheses in the format of $(x,y,z)$.   Within Appendix~\ref{app:subdata}, points are bounded by the set brackets in the format of $\{x,y,z\}$, as the required syntax within Mathematica, where many of the computations were performed.  For integrity of the data, it was judged best to report the points using this alternative formatting.

The code used and its output will be posted for public access, with a permanent site being determined.

\pagebreak


\begin{center}
{\bf Table I: Subdivision Points from 4 Iterations on $K_0$}
\end{center}

\vspace*{4ex}
\begin{multicols}{2}
\input{h4Un-TJP2a.txt}
\end{multicols}


\vspace*{3ex}

\begin{center}
\vspace*{-0.09in}
{\bf Table II: Subdivision Points from 4 Iterations on $K_1$}
\vspace*{-0.09in}
\end{center}

\vspace*{1ex}

\begin{multicols}{2}
\input{HW4-2.txt}
\end{multicols}

\pagebreak

\begin{center}
\vspace*{-0.09in}
{\bf Table III: Subdivision Points from 3 Iterations on $K_1$}
\vspace*{-0.09in}
\end{center}

\begin{multicols}{2}
\input{n3.tex}

\end{multicols}

\begin{center}
\vspace*{-0.09in}
{\bf Table IV: Subdivision Points from First Two Iterations on $K_1$}
\vspace*{-0.09in}
\end{center}

\begin{multicols}{2}
\input{n1.tex}

\end{multicols}

\pagebreak

\begin{multicols}{2}
\input{n2.tex}

\end{multicols}

\end{document}

%% file: h4Un-TJP2a.txt
\small{
\begin{verbatim}
     P[0]: Y,Z
     {          0, 4831838208,10737418240 },
     { -503316480, 1342177280,8388608000 },
     { -859832320, -1562378240,6249512960 },
     { -1092485120, -3959685120,4313317376 },
     { -1221918720, -5918269440,2572288000 },
     { -1266603520, -7498398720,1017953280 },
     { -1242964800, -8753032512,-358732160 }

     P[1]:  X,Y,Z
     { -1242964800, -8753032512,-358732160},
     { -1219326080, -10007666304,-1735417600},
     { -1127363840, -10936804608,-2934453760},
     { -983503360, -11593406976,-3964886016},
     { -802124800, -12023124992,-4836333568},
     { -595691520, -12265252864,-5558824960},
     { -374886400, -12353556480,-6142648320}

     P[2]: X
     { -374886400, -12353556480,-6142648320},
     { -154081280, -12441860096,-6726471680},
     {   81095680, -12376339456,-7171627008},
     {  319961600, -12190760448,-7488402432},
     {  553614080, -11913360640,-7687345664},
     {  774794880, -11567678336,-7779114240},
     {  977745600, -11173260096,-7774340480}

     P[3]; X,Y,Z
     {  977745600, -11173260096,-7774340480},
     { 1180696320, -10778841856,-7769566720 },
     { 1365416960, -10335687680,-7668250624 },
     { 1526149120, -9863344128,-7481024512 },
     { 1658634240, -9377366016,-7218495488 },
     { 1759969280, -8890023936,-6891110400 },
     { 1828454400, -8410890240,-6509035520 }

     P[4]: Y,Z
     { 1828454400, -8410890240,-6509035520 },
     { 1896939520, -7931756544,-6126960640 },
     { 1932574720, -7460831232,-5690195968 },
     { 1933660160, -7007686656,-5208907776 },
     { 1899699200, -6579196928,-4692981760 },
     { 1831246080, -6180123904,-4151902720 },
     { 1729745600, -5813581632,-3594648960 }

     P[5]: X,Y,Z
     { 1729745600, -5813581632,-3594648960 },
     { 1628245120, -5447039360,-3037395200 },
     { 1493697280, -5113027840,-2463966720 },
     { 1327546880, -4814661120,-1883341824 },
     { 1132129280, -4553405440,-1303992320 },
     {  910510080, -4329543680,-733777920 },
     {  666316800, -4142518272,-179855360 }

     P[6]: X,Y,Z
     {  666316800, -4142518272,-179855360 },
     {  422123520, -3955492864, 374067200 },
     {  155356160, -3805303808, 911697920 },
     { -130357760, -3691393536,1425880064 },
     { -430828800, -3612364032,1910159872 },
     { -741473920, -3566319744,2358877440 },
     { -1057492800, -3551088960,2767242880 }

     P[7]: X,Z
     { -1057492800, -3551088960,2767242880 },
     { -1373511680, -3535858176,3175608320 },
     { -1694904320, -3551440896,3543621632 },
     { -2016870400, -3595665408,3866492928 },
     { -2334392320, -3666083840,4140302336 },
     { -2642411520, -3760193536,4362076160 },
     { -2936012800, -3875536896,4529848320 }

     P[8]: X,Y
     { -2936012800, -3875536896,4529848320 },
     { -3229614080, -3990880256,4697620480 },
     { -3508797440, -4127457280,4811390976 },
     { -3768647680, -4282810368,4869193728 },
     { -4004392960, -4454526976,4870070272 },
     { -4211589120, -4640339456,4814131200 },
     { -4386312000, -4838103360,4702602880 }

     P[9]: Y,Z
     { -4386312000, -4838103360,4702602880 },
     { -4561034880, -5035867264,4591074560 },
     { -4703284480, -5245582592,4423956992 },
     { -4809136640, -5465104896,4202476544 },
     { -4875187200, -5692412928,3928975360 },
     { -4898744320, -5925586944,3606958080 },
     { -4878028800, -6162665472,3241123840 }

     P[10]: X,Y,Z
     { -4878028800, -6162665472,3241123840 },
     { -4857313280, -6399744000,2875289600 },
     { -4792325120, -6640727040,2465638400 },
     { -4681285120, -6883653120,2016869376 },
     { -4523326720, -7126519040,1534876160 },
     { -4318696320, -7367136640,1026778880 },
     { -4068961600, -7602868032, 500941440 }

     P[11]: X,Y,Z
     { -4068961600, -7602868032, 500941440},
     { -3819226880, -7838599424, -24896000},
     { -3524387840, -8069444608,-568473600},
     { -3186012160, -8292765696,-1121427456},
     { -2806988800, -8505475072,-1674149888},
     { -2391736320, -8703770624,-2215772160},
     { -1946419200, -8882749440,-2734161920}

     P[12]: X,Y,Z
     { -1946419200, -8882749440,-2734161920},
     { -1501102080, -9061728256,-3252551680 },
     { -1025720320, -9221390336,-3747708928 },
     { -526438400, -9356832768,-4207501312 },
     {  -11166720, -9462051840,-4618532864 },
     {  510222080, -9529556736,-4966141440 },
     { 1025668800, -9549861696,-5234410880 }

     P[13]: X
     { 1025668800, -9549861696,-5234410880 },
     { 1541115520, -9570166656,-5502680320 },
     { 2050620160, -9543271680,-5691610624 },
     { 2542123520, -9459691008,-5785285632 },
     { 3001379840, -9307943936,-5766535168 },
     { 3411732480, -9074046976,-5616947200 },
     { 3753881600, -8740884480,-5316894720 }

     P[14]: Y,Z
     { 3753881600, -8740884480,-5316894720 },
     { 4096030720, -8407721984,-5016842240 },
     { 4369976320, -7975293952,-4566325248 },
     { 4556418560, -7426484736,-3945716736 },
     { 4633414400, -6741046528,-3134174720 },
     { 4576145280, -5894969984,-2109669120 },
     { 4356676800, -4859733312,-849023360 }

     P[15]: X,Y,Z
     { 4356676800, -4859733312,-849023360 },
     { 4137208320, -3824496640, 411622400 },
     { 3755540480, -2600099840,1908408320 },
     { 3183738880, -1158021120,3664510976 },
     { 2390753280,  534773760,5704253440 },
     { 1342177280, 2516582400,8053063680 },
     {          0, 4831838208,10737418240 }
\end{verbatim}
}

%% file: HW4-2.txt
\small{
\begin{verbatim}
Q[0]: Y,Z
{          0, 4831838208,10737418240 },
{ -503316480, 1342177280,8388608000 },
{ -859832320, -1562378240,6249512960 },
{ -1089863680, -3959685120,4312924160 },
{ -1212088320, -5918269440,2570813440 },
{ -1243563520, -7498398720,1014497280 },
{ -1199764800, -8753032512,-365212160 }

Q[1]: X, Y, Z
{ -1199764800, -8753032512,-365212160 },
{ -1155966080, -10007666304,-1744921600 },
{ -1036893440, -10936804608,-2948024320 },
{ -858022400, -11593406976,-3983708160 },
{ -633246720, -12023124992,-4861665280 },
{ -374917120, -12265252864,-5591941120 },
{  -93900800, -12353556480,-6184796160 }

Q[2]: X
{  -93900800, -12353556480,-6184796160 },
{  187115520, -12441860096,-6777651200 },
{  490818560, -12376339456,-7233085440 },
{  806341120, -12190760448,-7561359360 },
{ 1124299520, -11913360640,-7772948480 },
{ 1436734080, -11567678336,-7878405120 },
{ 1737028800, -11173260096,-7888232960 }

Q[3]: X, Y, Z
{ 1737028800, -11173260096,-7888232960 },
{ 2037323520, -10778841856,-7898060800 },
{ 2325478400, -10335687680,-7812259840 },
{ 2594877440, -9863344128,-7641333760 },
{ 2840248320, -9377366016,-7395737600 },
{ 3057582080, -8890023936,-7085752320 },
{ 3244032000, -8410890240,-6721372160 }

Q[4]: X, Y, Z
{ 3244032000, -8410890240,-6721372160 },
{ 3430481920, -7931756544,-6356992000 },
{ 3586048000, -7460831232,-5938216960 },
{ 3707883520, -7007686656,-5475041280 },
{ 3794304000, -6579196928,-4977172480 },
{ 3844686080, -6180123904,-4453918720 },
{ 3859345600, -5813581632,-3914088960 }

Q[5]: Y, Z
{ 3859345600, -5813581632,-3914088960 },
{ 3874005120, -5447039360,-3374259200 },
{ 3852942080, -5113027840,-2817853440 },
{ 3796472320, -4814661120,-2253680640 },
{ 3705850880, -4553405440,-1690050560 },
{ 3583150080, -4329543680,-1134673920 },
{ 3431116800, -4142518272,-594575360 }

Q[6]: X, Y, Z
{ 3431116800, -4142518272,-594575360 },
{ 3279083520, -3955492864, -54476800 },
{ 3097717760, -3805303808, 470343680 },
{ 2889766400, -3691393536, 972861440 },
{ 2658650880, -3612364032,1446737920 },
{ 2408324480, -3566319744,1886407680 },
{ 2143108800, -3551088960,2287152640 }

Q[7]: X, Z
{ 2143108800, -3551088960,2287152640 },
{ 1877893120, -3535858176,2687897600 },
{ 1597788160, -3551440896,3049717760 },
{ 1307115520, -3595665408,3367895040 },
{ 1010565120, -3666083840,3638558720 },
{  713031680, -3760193536,3858759680 },
{  419430400, -3875536896,4026531840 }

Q[8]: X, Y
{  419430400, -3875536896,4026531840 },
{  125829120, -3990880256,4194304000 },
{ -163840000, -4127457280,4309647360 },
{ -444661760, -4282810368,4370595840 },
{ -711700480, -4454526976,4376166400 },
{ -960184320, -4640339456,4326420480 },
{ -1185710400, -4838103360,4222512640 }

Q[9]: X, Y, Z
{ -1185710400, -4838103360,4222512640 },
{ -1411236480, -5035867264,4118604800 },
{ -1613804800, -5245582592,3960535040 },
{ -1789012480, -5465104896,3749457920 },
{ -1932825600, -5692412928,3487621120 },
{ -2041784320, -5925586944,3178414080 },
{ -2113228800, -6162665472,2826403840 }

Q[10]: Y, Z
{ -2113228800, -6162665472,2826403840 },
{ -2184673280, -6399744000,2474393600 },
{ -2218603520, -6640727040,2079580160 },
{ -2212359680, -6883653120,1646530560 },
{ -2164081920, -7126519040,1180989440 },
{ -2072936320, -7367136640, 689914880 },
{ -1939361600, -7602868032, 181501440 }

Q[11]: X, Y, Z
{ -1939361600, -7602868032, 181501440 },
{ -1805786880, -7838599424,-326912000 },
{ -1629783040, -8069444608,-852664320 },
{ -1411788800, -8292765696,-1387560960 },
{ -1153515520, -8505475072,-1922170880 },
{ -858193920, -8703770624,-2445803520 },
{ -530841600, -8882749440,-2946498560 }

Q[12]: X, Y, Z
{ -530841600, -8882749440,-2946498560 },
{ -203489280, -9061728256,-3447193600 },
{  155893760, -9221390336,-3924951040 },
{  542289920, -9356832768,-4367810560 },
{  948894720, -9462051840,-4762542080 },
{ 1366849280, -9529556736,-5094635520 },
{ 1784952000, -9549861696,-5348303360 }

Q[13]: X
{ 1784952000, -9549861696,-5348303360 },
{ 2203054720, -9570166656,-5601971200 },
{ 2621305600, -9543271680,-5777213440 },
{ 3028503040, -9459691008,-5858242560 },
{ 3411102720, -9307943936,-5827993600 },
{ 3752929280, -9074046976,-5668126720 },
{ 4034867200, -8740884480,-5359042560 }

Q[14]: Y, Z
{ 4034867200, -8740884480,-5359042560 },
{ 4316805120, -8407721984,-5049958400 },
{ 4538854400, -7975293952,-4591656960 },
{ 4681899520, -7426484736,-3964538880 },
{ 4723884800, -6741046528,-3147745280 },
{ 4639505280, -5894969984,-2119173120 },
{ 4399876800, -4859733312,-855503360 }

Q[15]: X, Y, Z
{ 4399876800, -4859733312,-855503360 },
{ 4160248320, -3824496640, 408166400 },
{ 3765370880, -2600099840,1906933760 },
{ 3186360320, -1158021120,3664117760 },
{ 2390753280,  534773760,5704253440 },
{ 1342177280, 2516582400,8053063680 },
{          0, 4831838208,10737418240 }
\end{verbatim}
}

%% file: n3.tex
\small{

$K[1,3,0]$: Y, Z
\begin{verbatim}
{ 0, 4831838208,10737418240 },
{ -1006632960, -2147483648,6039797760 },
{ -1426063360, -6786383872,2181038080 },
{ -1420820480, -9707716608,-893386752 },
{ -1127219200, -11385044992,-3252682752 },
{ -655933440, -12176949248,-4975001600 },
{ -93900800, -12353556480,-6142648320 }
\end{verbatim}

$K[1,3,1]$: X
\begin{verbatim}
{ -93900800, -12353556480,-6142648320 },
{ 468131840, -12530163712,-7310295040 },
{ 1120911360, -12091473920,-7923269632 },
{ 1777500160, -11307614208,-8063877120 },
{ 2374696960, -10360258560,-7818575872 },
{ 2871132160, -9369157632,-7273185280 },
{ 3244032000, -8410890240,-6509035520 }
\end{verbatim}

$K[1,3,2]$: Y, Z
\begin{verbatim}
{ 3244032000, -8410890240,-6509035520 },
{ 3616931840, -7452622848,-5744885760 },
{ 3866296320, -6527188992,-4761976832 },
{ 3969351680, -5711167488,-3641638912 },
{ 3921920000, -5037965312,-2460712960 },
{ 3735183360, -4516569088,-1287700480 },
{ 3431116800, -4142518272,-179855360 }
\end{verbatim}

$K[1,3,3]$: X, Z
\begin{verbatim}
{ 3431116800, -4142518272,-179855360 },
{ 3127050240, -3768467456,927989760 },
{ 2705653760, -3541762048,1970667520 },
{ 2188902400, -3457941504,2890924032 },
{ 1609564160, -3499098112,3642753024 },
{ 1006632960, -3644850176,4194304000 },
{ 419430400, -3875536896,4529848320 }
\end{verbatim}

$K[1,3,4]$: X, Y
\begin{verbatim}
{ 419430400, -3875536896,4529848320 },
{ -167772160, -4106223616,4865392640 },
{ -739246080, -4421844992,4984930304 },
{ -1255669760, -4802740224,4872732672 },
{ -1677393920, -5229969408,4529192960 },
{ -1970339840, -5688508416,3972792320 },
{ -2113228800, -6162665472,3241123840 }
\end{verbatim}

$K[1,3,5]$: Y, Z
\begin{verbatim}
{ -2113228800, -6162665472,3241123840 },
{ -2256117760, -6636822528,2509455360 },
{ -2248949760, -7126597632,1602519040 },
{ -2070446080, -7616299008,557907968 },
{ -1712128000, -8089567232,-567672832 },
{ -1185546240, -8524791808,-1697382400 },
{ -530841600, -8882749440,-2734161920 }
\end{verbatim}

$K[1,3,6]$: X
\begin{verbatim}
{ -530841600, -8882749440,-2734161920 },
{ 123863040, -9240707072,-3770941440 },
{ 906690560, -9521397760,-4714790912 },
{ 1777500160, -9685598208,-5468651520 },
{ 2667560960, -9676472320,-5915246592 },
{ 3470991360, -9407209472,-5916999680 },
{ 4034867200, -8740884480,-5316894720 }
\end{verbatim}

$K[1,3,7]$
\begin{verbatim}
{ 4034867200, -8740884480,-5316894720 },
{ 4598743040, -8074559488,-4716789760 },
{ 4923064320, -7011172352,-3514826752 },
{ 4854906880, -5413797888,-1553989632 },
{ 4194304000, -3095396352,1342177280 },
{ 2684354560, 201326592,5368709120 },
{ 0, 4831838208,10737418240 }
\end{verbatim}
}

%% file: n1.tex
\small{
\begin{verbatim}
(*Subdivision 1 *)
{ 0, 4831838208,10737418240 },
{ -4026531840, -23085449216,-8053063680 },
{ 1342177280, -13555990528,-13421772800 },
{ 5704253440, -6442450944,-8858370048 },
{ 5368709120, -3388997632,-1610612736 },
{ 2768240640, -2952790016,3187671040 },
{ 419430400, -3875536896,4529848320 }

{ 419430400, -3875536896,4529848320 },
{ -1929379840, -4798283776,5872025600 },
{ -4026531840, -7079985152,3758096384 },
{ -3355443200, -9462349824,-2818572288 },
{ 2684354560, -10871635968,-10737418240 },
{ 10737418240, -13690208256,-10737418240 },
{ 0, 4831838208,10737418240 }
\end{verbatim}
}

%% file: n2.tex
\small{
\begin{verbatim}

(*Subdivision 2 *)
{ 0, 4831838208,10737418240 },
{ -2013265920, -9126805504,1342177280 },
{ -1677721600, -13723762688,-4697620480 },
{ -293601280, -13941866496,-7818182656 },
{ 1258291200, -12375293952,-8690597888 },
{ 2498232320, -10327425024,-8037335040 },
{ 3244032000, -8410890240,-6509035520 }

{ 3244032000, -8410890240,-6509035520 },
{ 3989831680, -6494355456,-4980736000 },
{ 4241489920, -4709154816,-2577399808 },
{ 3816816640, -3667918848,50331648 },
{ 2831155200, -3292528640,2323644416 },
{ 1593835520, -3414163456,3858759680 },
{ 419430400, -3875536896,4529848320 }

{ 419430400, -3875536896,4529848320 },
{ -754974720, -4336910336,5200936960 },
{ -1866465280, -5138022400,5007998976 },
{ -2600468480, -6121586688,3825205248 },
{ -2637168640, -7141851136,1784676352 },
{ -1840250880, -8166834176,-660602880 },
{ -530841600, -8882749440,-2734161920 }

{ -530841600, -8882749440,-2734161920 },
{ 778567680, -9598664704,-4807720960 },
{ 2600468480, -10005512192,-6509559808 },
{ 4613734400, -9789505536,-7063207936 },
{ 6039797760, -8355053568,-5368709120 },
{ 5368709120, -4429185024,0 },
{ 0, 4831838208,10737418240 }
\end{verbatim}
}